\documentclass[10pt]{amsart}
\usepackage{parskip}
\usepackage{mathrsfs}
\usepackage{graphicx}
\usepackage{cite}

\title{ M\"{o}bius Polynomials of Face Posets of Convex Polytopes}
\author{Meena Jagadeesan, Susan Durst}
\date{\today}

\DeclareMathAlphabet{\mathpzc}{OT1}{pzc}{m}{it}

\theoremstyle{theorem}
\newtheorem{prop}{Proposition}
\newtheorem{theorem}[prop]{Theorem}
\newtheorem{lemma}[prop]{Lemma}
 
\newtheorem{cor}[prop]{Corollary}

\theoremstyle{definition}
\newtheorem{definition}[prop]{Definition}

\begin{document}

\maketitle

\begin{abstract}

The M\"obius polynomial is an invariant of ranked posets, closely related to the M\"obius function.  In this paper, we study the M\"obius polynomial of face posets of convex polytopes.  We present formulas for computing the M\"obius polynomial of the face poset of a pyramid or a prism over an existing polytope, or of the gluing of two or more polytopes in terms of the M\"obius polynomials of the original polytopes.  We also present general formulas for calculating M\"obius polynomials of face posets of simplicial polytopes and of Eulerian posets in terms of their $f$-vectors and some additional constraints.\end{abstract} 

\section{Introduction}

A \textit{ranked poset} is a poset $P$, together with a rank function $|\cdot|: P\rightarrow \mathbb{N}$, such that whenever $p>q$ and for any $p\geq r\geq q$ we have $p=r$ or $q=r$, we have $|p|=|q|+1$.  Given a finite ranked poset $P$, we define the \textbf{M\"obius polynomial} of $P$ to be
\[\mathcal{M}_P(z)=\sum_{p\leq q\in P}\mu(p,q)z^{|q|-|p|},\]
where $\mu$ is the M\"obius function defined on $P$.

Our interest in the M\"obius polynomial stems from previous work on the splitting algebra or universal labeling algebra of a poset, introduced by Gelfand, Retakh, Serconek, and Wilson in their 2005 paper \cite{AlgebrasAssocToDirGraphs}.  In their 2007 paper \cite{HilbertSeriesPaper}, they showed that the Hilbert series of the algebra $A(P)$ associated to a finite ranked poset $P$ could be calculated using the formula:
\[H(A(P),z)=\frac{1-z}{1-z\mathcal{M}_P(z)}.\]
Using this fact, they were able to calculate the Hilbert series of the algebra associated to the Boolean lattice of order $n$, and to the poset of subspaces of a finite-dimensional vector space over a finite field.  

In 2009, Duffy used the same result to calculate the Hilbert series of the algebra associated to the face poset of an $n$-gon.  She also found that the graded trace generating function for the automorphism of $A(P)$ induced by an automorphism $\sigma$ of the poset $P$ is given by
\[Tr_\sigma(A(P),z)=\frac{1-z}{1-z\mathcal{M}_{P^\sigma}(z)},\]
where $P^\sigma$ is the poset consisting of fixed points of $P$ under the automorphism $\sigma$.

In their 2009 paper \cite{GenLayeredGraphs}, Retakh and Wilson generalize their Hilbert series result to posets which have a weaker rank function $\rho$, which only satisfies the property that if $p<q$, then $\rho(p)<\rho(q)$.  In this setting, they define 
\[\mathcal{M}_P(z)=\sum_{p\leq q\in P}\mu(p,q)z^{\rho(q)-\rho(p)},\]
and they present a generalization of the universal labeling algebra, which they also call $A(P)$, satisfying the same Hilbert series equation.  This paper also contains formulas for calculating the M\"obius polynomial of a poset under various operations, including adding intermediate elements, adding additional relations between existing elements, and joining two posets by identifying their maximal and/or minimial vertices.  In 2014, Durst showed \cite{DirectProductsPaper} that given two posets $P$ and $Q$, the M\"obius polynomial of their direct product $P\times Q$ is given by
\[\mathcal{M}_{P\times Q}(z)=\mathcal{M}_P(z)\mathcal{M}_Q(z),\]
and used this to calculate the Hilbert series and graded trace generating functions associated to the poset of factors of a natural number $n$ ordered by divisibility.  

In this paper we continue the study of M\"obius polynomials, focusing our attention on the face posets of convex polytopes.  The organization of this paper is as follows. In Section 2, we present basic definitions. In Section 3, we present a result regarding the sum of the coefficients of the M\"{o}bius polynomials. In Section 4, we recall known results  about face posets and include a proof that every interval of a face poset is itself a face poset. In Section 5, we proceed to compute the M\"{o}bius polynomial of the face poset of a pyramid over a polytope terms of the M\"{o}bius polynomial of the face poset of the original polytope. In Section 6, we compute the M\"{o}bius polynomial of the face poset of a prism over a polytope in terms of the M\"{o}bius polynomial of the face poset of the original polytope and its bottom polynomial, which we will define in Section 3. In Section 7, we compute the M\"{o}bius polynomial of the face poset of polytopes glued together, in terms of the  M\"{o}bius polynomial of the original polytopes and their top polynomials, which we will define in Section 3. In Section 8, we compute the M\"{o}bius polynomial of the face poset of a simplicial polytope in terms of its $f$-vector. In Section 9, we compute the M\"{o}bius polynomial of an Eulerian poset in terms of its $f$-vector and some additional constraints. This gives us a general formula for computing the M\"{o}bius polynomial of the face poset of an arbitrary convex polytope.

\section{Definitions}
\begin{definition}
A ranked poset is a poset $P$, together with a rank function $| \cdot|: P \rightarrow \mathbb{N}$ satisfying the following properties:
\begin{itemize}
\item[(i)] If $p > q$, and for any $p \ge r \ge q$ we have $r=p$ or $r=q$, then $|p| = |q| + 1$. 
\item[(ii)] If $p$ is minimal in $P$, then $|p| = -1$.
\end{itemize}
\end{definition}

\begin{definition}
In a ranked poset $P$, for each $p \in P$, we call $|p|$ the rank of $p$. We use the following notation:
\begin{align*}
P_i &= \left\{ p \in P \mid |p| =i \right\} \\
P_{\ge i} &= \left\{ p \in P \mid |p| \ge i \right\} \\
P_{\le i} &= \left\{ p \in P \mid |p| \le i \right\}
\end{align*}
We define the rank of a poset $P$ to be $\max\left\{|p| \mid p \in P \right\}$. 
\end{definition}

Notice that this is not the standard definition of rank. In the standard definition, minimal elements have rank $0$. We have set our minimum rank to $-1$ so that rank will correspond to dimension in face posets of convex polytopes. 

\begin{definition}
Given $p, q \in P$ we define the closed interval $[p, q] \in P$ by 
\[[p,q] = \left\{s \in P \mid p \le s \le q \right\}.\]
 If $p \not\leq q$, then $[p ,q]$ is defined to be the empty set.  We define $I(P)$ to be the set of all closed intervals of the poset $P$. 
\end{definition}

\begin{definition}
The M\"{o}bius function $\mu^P: I(P) \rightarrow \mathbb{N}$ is defined by
\begin{itemize}

\item[(i)] $ \mu^P [p, p] = 1$ for $p \in P$\\
\item[(ii)]$\mu^P [p, q] = - \sum_{p \le s < q} \mu^P [p, s]$ for $p < q \in P$ \\
\item[(iii)]$\mu^P [p, q] = 0$ for $p \not\leq q \in P$.
\end{itemize}
\end{definition}
Notice that 
$$\displaystyle \sum_{p \le s < q} \mu^P [p, s] = - \sum_{p < s \le q} \mu^P [s, q],$$
so 
$$\displaystyle \mu^P [p, q] = - \sum_{p < s \le q} \mu^P [s, q].$$
For ease of notation, we will often write $\mu$ rather than $\mu^P$.

\begin{definition}
Given a ranked poset $P$, we define the polynomial M\"{o}bius function $\mu_z^P: I(P) \rightarrow \mathbb{N} (z)$ by 
$$\displaystyle \mu_z^P [p, q] = \mu^P [p, q] z^{|q| - |p|}.$$
\end{definition}
For ease of notation, we will often write $\mu_z$ rather than $\mu_z^P$. Notice that $\mu_1 [p, q] = \mu [p, q]$. 
\begin{definition}
The M\"{o}bius polynomial $\mathcal{M}_P(z)$ of a poset $P$ is given by 
$$\displaystyle \mathcal{M}_P(z)= \sum_{p \le q \in P} \mu_z [p,q].$$ 
\end{definition}

\begin{definition}
Given posets $P$ and $Q$, we define the direct product of $P$ and $Q$ to be
$$\displaystyle P \times Q = \left\{(p, q) \mid p \in P, q \in Q \right\}$$
with ordering relation $\le_{P \times Q}$  defined so that
$(p_1, q_1) \le_{P \times Q} (p_2, q_2)$ if and only if $p_1 \le_P p_2$ and $q_1 \le_Q q_2$. 
\end{definition}

We will use the following result from \cite{DirectProductsPaper}: 

\begin{lemma}
\label{dp}
Given posets $P$ and $Q$ with direct product $P \times Q$, 
$$\displaystyle \mathcal{M}_{P \times Q}(z) = \mathcal{M}_P(z) \cdot \mathcal{M}_Q (z).$$
\end{lemma}

\section{Coefficients of M\"obius Polynomials}

\begin{definition}
Let $-1 \le r \le |P|$. We define the $r$-polynomial $g_r^P(z)$ of the poset $P$ to be \[\sum_{|p| \le r \le |q|} \mu_z^P [p, q].\] We refer to $g_{|P|}^P(z)$ as the top polynomial and $g_{-1}^P(z)$ as the bottom polynomial.
\end{definition}

\begin{lemma}
\label{Muiszero}
For any $p, q \in P$ 
$$\displaystyle \sum_{p \le s \le q} \mu [p, s] = \sum_{p \le s \le q} \mu [s, y] = 0.$$ 
\end{lemma}

\begin{proof}
By definition, we have $$\displaystyle \mu [p, q] = -\sum_{p < s \le y} \mu [s, q]$$ \\
This implies $$\displaystyle\sum_{p \le s \le q} \mu [s, q] = 0.$$ \\
We also have $$\displaystyle\mu [p, q] = -\sum_{p \le s < q} \mu [s, q],$$ 
which implies $$\displaystyle\sum_{p \le s \le q} \mu [p, s] = 0.$$ \\
\end{proof}

\begin{lemma}
\label{unmaxmin}
If a poset P has a unique maximum element or unique minimum element, then $\mathcal{M}_P (1) = 1$. 
\end{lemma}

\begin{proof}
First we will prove the result for a poset $P$ with unique maximum element $\hat{1}$. We proceed by induction on $|P|$. If $|P| = -1$, then 
$$\displaystyle \mathcal{M}_P (z) = \mu_z [\hat{1}, \hat{1}] = 1,$$
and it follows that $\mathcal{M}_P (z) = 1$. \\
Now assume that $|P| \ge 0$, and that the result holds for posets with smaller rank. No interval in $P$ can contain more than one element of $P_{-1}$ since there do not exist $p, q \in P_{-1}$ with $p < q$. Hence $g_{-1}^P (z)$ is the sum of the generalized M\"{o}bius functions of all intervals in $P$ that contain one of the $p$ in $P_{-1}$ with no interval counted more than once. Thus
$$\displaystyle \mathcal{M}_P (1) = g_{-1}^P (1) + \mathcal{M}_{P_{\ge 0}} (1).$$
We have 
\begin{align*}
g_{-1}^P (1) &= \sum_{|p| \le -1 \le |q|} \mu_z [p, q] \\
&= \sum_{p \in P_{-1}} \sum_{p \le q \le \hat{1}} \mu_z [p, q].
\end{align*}
This is equal to $0$ by Lemma~\ref{Muiszero}. \\
This gives us 
$$\displaystyle \mathcal{M}_P (1) = \mathcal{M}_{P_{\ge 0}} (1).$$
By the induction hypothesis, $\mathcal{M}_{P_{\ge 0}} (1) = 1$. Thus for any $P$ with a unique maximum  element, $\mathcal{M}_P (1) = 1$. The proof for $P$ with a unique minimum element is completely analogous. 
\end{proof}

\begin{definition}
For a  poset $P$, we define the poset $\tilde{P}$ to be $P \cup \left\{\hat{0} \right\} \cup \left\{\hat{1} \right\}$ with ordering relation $\le_{\tilde{P}}$ given by: \\
For $p, q \in P$, $p \le_{\tilde{P}} q$ if $p \le_P q$. \\
For $p \in \tilde{P}$, $p \le_{\tilde{P}} {\hat{1}}$ and $\hat{0}\le_{\tilde{P}} p$.
\end{definition} 

\begin{lemma}
For any poset $P$, 
$\mathcal{M}_P (1)  = \mu^{\tilde{P}} [\hat{0},  \hat{1}] + 1.$ 
\end{lemma}

\begin{proof}
For any $p \in \tilde{P}$, we have $\hat{0} \le p \le \hat{1}.$  
Hence,
$$\displaystyle \mathcal{M}_{\tilde{P}} (1) = \sum_{\hat{0} \le p \le \hat{1}} \mu [\hat{0}, p] + \sum_{\hat{0} \le p \le \hat{1}} \mu [p, \hat{1}] - \mu [\hat{0}, \hat{1}] + \mathcal{M}_P (1).$$
By Lemma~\ref{Muiszero}, 
$$\displaystyle \sum_{\hat{0} \le p \le \hat{1}} \mu [\hat{0}, p] = \sum_{\hat{0} \le p \le \hat{1}} \mu [p, \hat{1}] = 0.$$
This means that
$$\displaystyle \mathcal{M}_{\tilde{P}} (1) = \mathcal{M}_P (1)  - \mu [\hat{0}, \hat{1}].$$

By Lemma~$\ref{unmaxmin}$, 
$$\displaystyle
 \mathcal{M}_{\tilde{P}} (1) = 1.$$

Hence, 
$$\displaystyle 1 = \mathcal{M}_P (1)  - \mu [\hat{0},  \hat{1}].$$

It follows that
$$\displaystyle \mathcal{M}_P (1)  = 1 + \mu [\hat{0},  \hat{1}].$$

\end{proof}

\section{Face Posets}

\begin{definition}
Given a  convex polytope $\mathpzc{P}$ with vertex set $V$, for each face $\mathpzc{p}$, we define $p_V$ to be the set of vertices contained in $\mathpzc{p}$. Hence, 
$$\displaystyle p_v = \left\{ v \in V \mid v \in \mathpzc{p}\right\}.$$
We define the face poset $P$ to be: 
$$\displaystyle \left\{p_v \mid \mathpzc{p} \in \mathpzc{P} \right\} \cup \left\{\emptyset \right\}$$
with ordering given by set inclusion. 
\end{definition}
For ease of notation, we will often write $p$ rather than $p_v$ in a face poset. \\
Notice that with our definition of the rank function, an $r$-dimensional face $\mathpzc{p}$ of a polytope $\mathpzc{P}$ corresponds to a rank $r$ element $p$ in the face poset $P$. 

\begin{definition}
In a $d$-dimensional convex polytope $\mathpzc{P}$, we define $\mathpzc{f_r}$ to be the number of $r$-dimensional faces of $\mathpzc{P}$ for $-1 \le r \le d$. We define $\mathpzc{f_{-1}}$ to be $1$. We call $\mathpzc{(f_{-1}, f_0,...,f_d)}$ the $\mathpzc{f}$-vector of $\mathpzc{P}$. 
\end{definition}

\begin{definition}
In an arbitrary poset $P$, we define $f_r$ to be the number of elements of rank $r$, $|P_r|$. We call $(f_{-1}, f_0,...,f_d)$ the $f$-vector of $P$. 
\end{definition}

For a convex polytope $\mathpzc{P}$ with face poset $P$, this means that $\mathpzc{f_r}$ represents the number of $r$-dimensional faces of $\mathpzc{P}$ and $f_r$ represents the number of elements of rank $r$ in its face poset $P$. Hence, the $\mathpzc{f}$-vector of $\mathpzc{P}$ is equal to the $f$-vector of $P$. 

\begin{definition}
Given a poset $P$, we define the dual poset $P^{*}$ to be the set $P$ with ordering relation as follows: $p \le_{P^{*}} q$ if and only if $q \le_P p$. 
\end{definition}
Note that the poset $P^{**}$ is isomorphic to the poset $P$.\\

By part (iv) of Theorem 2.7 in \cite{lec}, we know that the dual poset $P^{*}$ of a face poset $P$ of a convex polytope $\mathpzc{P}$ is itself a face poset of a convex polytope. We call this polytope the dual polytope of $\mathpzc{P}$, denoted by $\mathpzc{P^{*}}$.

\begin{theorem}
Any interval of a face poset $P$ of a convex polytope is itself a face poset of some convex polytope.
\end{theorem}

This result is part (ii) of Theorem 2.7 in \cite{lec}. Ziegler provides a geometric proof of this result. Here we present a combinatorial proof. 

\begin{proof}
Consider the interval $[p,q]$. Note that the interval $[\hat{0}^P, q]$ is the face poset of $\mathpzc{q}$. Call this subposet $Q$. Consider the dual poset $Q^{*}$. This corresponds to the face poset of the dual polytope $q^{*}$. Note that $\hat{0}^{Q^{*}} = q$. Since $p \le q$ in $Q$, $\hat{0}^{Q^{*}} = q \le p$ in $Q^{*}$. The interval $[\hat{0}^{Q^{*}}, p]$ is the face poset of the convex polytope into which $\mathpzc{p}$ gets transformed in the dual polytope $\mathpzc{q}^{*}$. Call this polytope $\mathpzc{s}$. Now, consider $Q^{**} = Q$. The interval $[p, q] \in Q$ is the face poset of the dual polytope $\mathpzc{s}^{*}$. 
\end{proof}

\begin{theorem}
\label{Faceuler}
In any face poset $P$, $\mu [p, q] = (-1)^{|q| - |p|}$
\end{theorem}

\begin{proof}
By the above theorem, it suffices to prove that $\mu [\hat{0}^P, \hat{1}^P] = (-1)^{d+1}]$ for all convex polytopes $P$ of dimension $d$ for $d \ge 0.$
We proceed by induction on $|q| - |p|$. \\
For our base cases, we know that $\mu [p,p] = 1$. Consider a convex polytope with dimension $d = 0$. Note that the face poset of this polytope is a totally ordered set of size 2. Hence $\mu [p,q] = -1$. \\
Now, assume that the statement holds for $d \le i$ We will show that it holds for $d = i+1$. 
Note that
$$\displaystyle \mu [\hat{0}^P, \hat{1}^P] = -\sum_{p \in P} \mu [\hat{0}^P, p].$$
By the induction hypothesis, 
\begin{align*}
 \mu [\hat{0}^P, \hat{1}^P]  &= -\sum_{r=-1}^k f_i * (-1)^{r+1} \\
 &= -[1 -\sum_{r=0}^i f_r * (-1)^{r}].
\end{align*}
Using the Euler characteristic of convex polytopes: 
\begin{align*}
 \mu [\hat{0}^P, \hat{1}^P] &= -[1 -((-1)^{i+2} + 1)] \\
  &= -[-(-1)^{i+2}] \\
 &= (-1)^{i+2}. 
\end{align*}
Hence $\mu_z [p, q] = (-z)^{|q| - |p|}$.
\end{proof}

\begin{lemma}
\label{facepoly}
In a poset $P$ with rank $d$, 
$$\displaystyle g_{|P|}^P (z) = \sum_{r=-1}^{d} f_r \cdot (-z)^{d-r}$$
$$\displaystyle g_{-1}^P (z) = \sum_{r=-1}^d f_r \cdot (-z)^{r+1}.$$
\end{lemma}

\begin{proof}
Note that
$$\displaystyle g_{|P|}^P (z) = \sum_{p \le \hat{1} \le q} \mu_z [p, q].$$
Since $p \le {\hat{1}}$ for all $p \in P$ and the only $q \in P$ such that $\hat{1} \le q$ is $q = \hat{1}$, 
$$\displaystyle g_{|P|}^P (z) = \sum_{p \in P} \mu_z [p,\hat{1}]$$
and
$$\displaystyle \mu_z [p, \hat{1}] = (-z)^{|\hat{1}| - |p|} = (-1)^{d - |p|}.$$
Hence,
$$\displaystyle g_{|P|}^P (z) = \sum_{|p|= r} (-z)^{d-r}.$$
Since there are $f_r$ elements in $P$ of rank $r$,
$$\displaystyle g_{|P|}^P (z) = \sum_{r=-1}^d f_r \cdot (-z)^{d-r}.$$
The proof is analogous for $g_{-1}^P$. 
\end{proof}

\section{Pyramids}
Given a $d$-dimensional convex polytope $\mathpzc{P}$, we can obtain the $(d+1)$-dimensional pyramid over $\mathpzc{P}$, denoted $\mathpzc{Py(P)}$, as follows:
Locate $\mathpzc{P}$ on the plane $x=0$ in $\mathbb{R}^{d+1}$. We construct an additional vertex $\alpha$ in the plane $x=1$ and take the union of lines connecting $\mathpzc{P}$ with $\alpha$. 

For each $r$-dimensional face $\mathpzc{p}$ in the polytope $\mathpzc{P}$, there exists a corresponding $(r+1)$ dimensional face $\mathpzc{p}'$ in $\mathpzc{Py(P)}$  connecting the vertex $\alpha$ with $\mathpzc{p}$. All faces in $\mathpzc{Py(P)}$ are of the form $\mathpzc{p}$ or $\mathpzc{p}'$. 

Figure~\ref{PyramidOverSquare} shows the pyramid construction applied to a square.

\begin{figure}[h]
\caption{}
\includegraphics{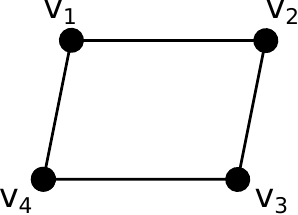} \hspace{.5in}\includegraphics{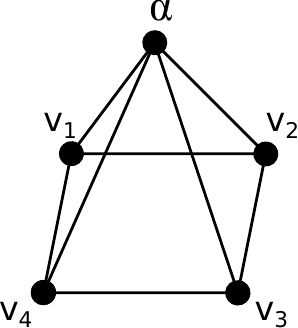}
\label{PyramidOverSquare}
\end{figure}

\begin{definition}
Given a face poset $P$, for each element $p \in P$, define $p'$ so that $p' = p \cup {\alpha}$. 
Let $\lambda_P = \left\{p' | p \in P\right\}$. We call $\lambda_P$ the pyramid addition poset of $\lambda_P$. 
We define the poset $Py(P)$ to be $P \cup \lambda_P$ where the relation is set inclusion. We call $Py(P)$ the pyramid of $P$. 
\end{definition}

Note that if $P$ is the face poset of $\mathpzc{P}$, then, by construction $Py(P)$ is the face poset of $\mathpzc{Py(P)}$. Figure~\ref{PyramidAddition} shows the face poset of a square and its corresponding pyramid addition poset---it is clear that the two posets are isomorphic.  Below, we will prove that this is always the case.

\begin{figure}[h]
\caption{}
\includegraphics{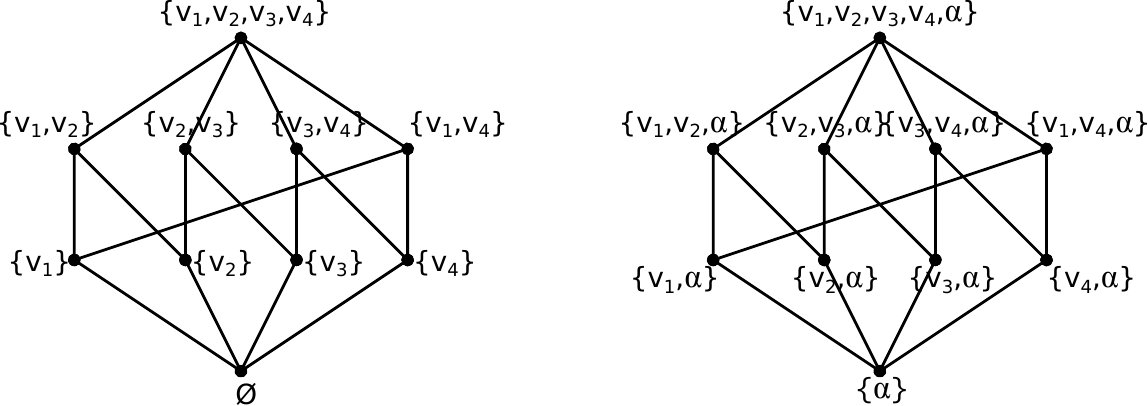}
\label{PyramidAddition}
\end{figure}

\begin{lemma}
\label{bij}
Let $P$ be a face poset with pyramid addition poset $\lambda_P$.Then  
$$\displaystyle \lambda_P \cong P.$$
\end{lemma}

\begin{proof}
Define a function 
\[\begin{array}{ccccc}
\phi&:&P&\rightarrow&\lambda_P\\
&&p&\mapsto&p'\\
\end{array}\]
We wish to show that $p \le_P q$ if and only if $p' \le_{\lambda_P} q'$. 
We know $p \le_P q$ if and only if $p \subseteq q$. Furthermore, $p \subseteq q$ if and only if $p \cup \left\{\alpha \right\} \subseteq q \cup \left\{\alpha\right\}$. Since $p' = p \cup \left\{\alpha \right\}$and $q' = q \cup \left\{\alpha\right\}$, we know that $q' = q \cup \left\{\alpha\right\}$ if and only if $p' \subseteq q'$. By definition, $p' \subseteq q'$ if and only if $p' \le_{\lambda_P}q'$. \\ Hence, $\phi$ defines an isomorphism of posets and $\lambda_P \cong P$. 
\end{proof}

\begin{definition}
Let $T_2 = \left\{t_1, t_2 \right\}$ be a totally ordered poset on two elements such that $t_1 \le_{T_2} t_2$.
\end{definition}

Figure~\ref{DirectProduct} shows the face poset of a pyramid with a square base.  This poset is the direct product of the face poset of the square with the poset $T_2$.  The diagram is broken down to emphasize this: the grey lines represent the original face poset, the dashed lines show the isomorphic pyramid addition poset, and the solid black lines represent the isomorphic copies of $T_2$ connecting them.   The face poset of a pyramid over a polytope will always have this direct product structure, as we will prove below.

\begin{figure}[h]
\caption{}
\label{DirectProduct}
\includegraphics{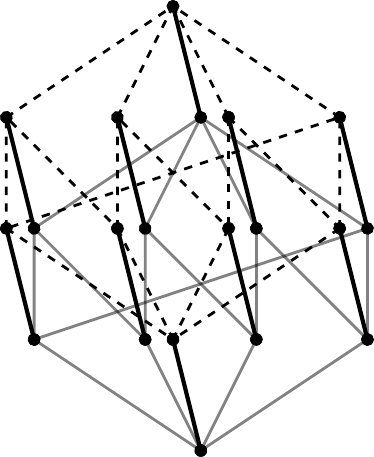}
\end{figure}

\begin{prop}
\label{pydp}
Let $P$ be a face poset with pyramid $Py(P)$. Then, 
$$\displaystyle Py(P) \cong P \times T_2.$$
\end{prop}

\begin{proof}
Define a function:   
\[\begin{array}{ccccc}
\psi&:&P \times T_2&\rightarrow&Py(P)\\
&&(p, t_1)&\mapsto&p\\
&&(p, t_2)&\mapsto&p'.\\
\end{array}\]
We wish to show that $(p, s) \le (q,t)$ if and only if $\psi((p,s)) \le \psi((q,t)).$

For $s = t = t_1$, we need to show that $(p, t_1) \le (q, t_1)$ if and only if $ \psi((p, t_1)) \le \psi ((q, t_1))$.  We know that $(p, t_1) \le_{P \times T_2} (q, t_1)$ if and only if $p \le_P q$. For $p, q \in P$, we have $p \le_P q$ if and only if $p \le_{Py(P)} q$. Since $\psi((p, t_1)) = p$ and $\psi((q, t_1)) = q$, we have$p \le_{Py(P)} q$ if and only if  $\psi((p, t_1)) \le_{Py(P)} \psi ((q, t_1))$. 

For $s = t = t_2$, we need to show that $(p, t_2) \le (q, t_2)$ if and only if $\psi((p, t_2)) \le \psi ((q, t_2))$. We know that $(p, t_2) \le_{P \times T_2} (q, t_2)$ if and only if $p \le_P q$. By Lemma~$\ref{bij}$, we have $p \le_P q$ if and only if $p' \le_{\lambda_P} q'$. We know that $p' \le_{\lambda_P}  q'$ if and only if $p' \le_{Py(P)} q'$. Since  $\psi((p, t_2)) = p'$ and $\psi((q, t_2)) = q'$, we know that $p' \le_{Py(P)} q'$ if and only if  $\psi((p, t_2)) \le_{Py(P)} \psi ((q, t_2))$. 

For $s = t_1 , t = t_2$, we need to show that $(p, t_1) \le (q, t_2)$ if and only if $\psi((p, t_1)) \le \psi((q, t_2))$. We know that $(p, t_1) \le_{P \times T_2} (q, t_2)$ if and only if $p \subseteq q$. Furthermore, we know that $p \subseteq q$ if and only if $p \subseteq q \cup \left\{\alpha\right\}$. Since $q' = q \cup \left\{\alpha\right\}$, we know that $p \subseteq q \cup \left\{\alpha\right\}$  if and only if $p \subseteq q'$. Since $p = \psi((p, t_1))$ and $\psi((q, t_2)) = q'$, we know that $p \subseteq q'$ if and only if $\psi((p, t_1)) \le \psi((q, t_2))$.

For $s = t_2, t = t_1$, we need to show that $(p, t_2) \le (q, t_1)$ if and only if $\psi((p, t_2)) \le \psi ((q, t_1))$. Since $t_2 > t_1$ there do not exist $p, q \in P$ such that $(p, t_2) \le_{P \times T_2} (q, t_1)$. We will now show there do not exist $p, q \in P$ such that $\psi((p, t_2)) \le \psi ((q, t_1))$. We know that $\psi((p, t_2)) \le \psi ((q, t_1))$ if and only if $p' \le q$. This holds if and only if $p \cup \left\{\alpha\right\} \subseteq q$. Since $q \cap \left\{\alpha\right\} = \emptyset$ for all $q \in P$, there do not exist $p, q \in P$ such that $p \cup \left\{\alpha\right\} \subseteq q$.

Hence, $\psi$ defines an isomorphism of posets and we have our result. 
\end{proof}

\begin{theorem}
\label{pyramid}
Given a $d$-dimensional convex polytope $\mathpzc{P}$ with face poset P, 
$$\displaystyle \mathcal{M}_{Py(P)} (z)= (2-z) \cdot \mathcal{M}_P (z).$$
\end{theorem}

\begin{proof}
By Proposition~$\ref{pydp}$ and Lemma~$\ref{dp}$, we know that
$$\displaystyle \mathcal{M}_{Py(P)} (z) = \mathcal{M}_P (z) \cdot \mathcal{M}_{T_2} (z) = (2-z) \cdot \mathcal{M}_P (z).$$
\end{proof}

\begin{cor}
\label{simplex}
If $\mathpzc{S}^d$ is a $d$-dimensional simplex with face poset $S^d$, then
$$\displaystyle \mathcal{M}_{S^d} (z) = (2-z)^{d+1}.$$ 
\end{cor}

\begin{proof}
We proceed by induction on $d$. For our base case, consider a simplex with dimension $d=0$. Notice that $S^0 \cong T_2$. Hence, $\mathcal{M}_{S^0} (z) = 2-z$. \\ 
Now, we assume that $d > 0$ and that the result holds for simplices of smaller dimension. Notice that $\mathpzc{S}^{d}$ is the $d$ dimensional pyramid over $\mathpzc{S}^{d-1}$. This means that
$$\displaystyle \mathcal{M}_{S^{d+1}} (z) = (2-z) \cdot (2-z)^{d+1} = (2-z)^{d+2}.$$ Thus, the result holds for any simplex $\mathpzc{S}^d$ with face poset $S^d$ for any $d \ge 0$. 
\end{proof}
\section{Prisms}
Given a $d$-dimensional convex polytope $\mathpzc{P}$, we can obtain the $(d+1)$-dimensional pyramid over $\mathpzc{P}$, denoted $\mathpzc{Pr(P)}$, as follows:Let  $\mathpzc{P}_1$ and $\mathpzc{P}_2$ be isomorphic copies of  $\mathpzc{P}$. Locate $\mathpzc{P}_1$ on the plane $x=0$ and $\mathpzc{P}_2$ on the plane $x=1$ in $\mathbb{R}^{d+1}$. We take the union of lines corresponding points on $\mathpzc{P_1}$ and $\mathpzc{P_2}$. Let $V$ be the vertex set of $\mathpzc{P}$, let $V_1$ be the vertex set of $\mathpzc{P}_1$, and let $V_2$ be the vertex set of $\mathpzc{P}_2$. 

For every $r$-dimensional face $\mathpzc{p}$ in the  polytope $\mathpzc{P}$, there exists a corresponding $r$-dimensional face $\mathpzc{p}_1$ in the polytope $\mathpzc{P}_1$, $r$-dimensional face $\mathpzc{p}_2$ in the  polytope $\mathpzc{P}_2$, and $(r+1)$-dimensional face $\mathpzc{p}'$ in $\mathpzc{Pr(P)}$ connecting $\mathpzc{p}_1$ and $\mathpzc{p}_2$. All faces in $\mathpzc{Pr(P)}$ are of the form $\mathpzc{p}'$,  $\mathpzc{p}_1$ or $\mathpzc{p}_2$. Figure~\ref{prism} shows the prism construction applied to a square. 

\begin{figure}[h]
\caption{}
\label{prism}
\includegraphics{Figure1} \hspace{.5in}\includegraphics{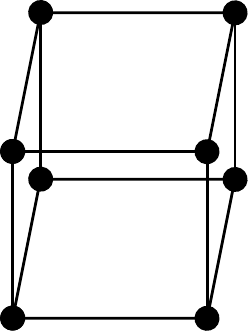}
\end{figure}

\begin{definition}
Given a face poset $P$ of $\mathpzc{P}$, let $P_1$ and $P_2$ be isomorphic copies of $P$, such that $P_1$ is the face poset of $\mathpzc{P}_1$ and $P_2$ is the face poset of $\mathpzc{P}_2$. Just as elements of $P$ are subsets of the vertex set $V$ of $\mathpzc{P}$, elements of $P_1$ are subsets of the vertex set $V_1$ of $\mathpzc{P}_1$ and elements of $P_2$ are subsets of the vertex set $V_2$ of $\mathpzc{P}_2$. For $p \in P$ corresponding to $p_1 \in P_1$ and $p_2 \in P_2$, define $p'$ so that $p' = p_1 \cup p_2$. 
Let $\beta_P =  \left\{p' | p \in P\right\}$. We call $\beta_P $ the prism addition poset of $P$. 
We define the prism of $P$ to be $Pr(P) = P_1 \cup P_2 \cup \beta_P$ with ordering relation given by set inclusion. 
\end{definition}

Note that if $P$ is the face poset of $\mathpzc{P}$, then by construction $Pr(P)$ is the face poset of $\mathpzc{Pr(P)}$.

\begin{lemma}
\label{bijj} 
Let $P$ be a face poset with prism addition poset $\beta_P $. Then
$$\displaystyle (\beta_P)_{\ge 0} \cong P_{\ge 0}.$$ 
\end{lemma}

\begin{proof}
We define a function 
\[\begin{array}{ccccc}
\phi&:&P'&\rightarrow&(\beta_P)_{\ge 0}\\
&&p&\mapsto&p'.\\
\end{array}\]
We wish to show that $\phi$ is an isomorphism of posets, so we need to show that $p \le_{P_{\ge 0}} q$ if and only if $p' \le_{(\beta_P)_{\ge 0}} q'$.  

We know that $p \le_{P_{\ge 0}} q$ if and only if $p \subseteq q$. This is equivalent to $p_1 \subseteq q_1$ and $p_2 \subseteq q_2$, or $p_1 \cup p_2 \subseteq q_1 \cup q_2$. Since $p' = p_1 \cup p_2$ and $q' = q_1 \cup q_2$, we know that $p_1 \cup p_2 \subseteq q_1 \cup q_2$ if and only if $p' \subseteq q'$. Furthermore, $p' \subseteq q'$ if and only if $p' \le_{(\beta_P)_{\ge 0}} q'$.  This gives us our result

\end{proof}

\begin{definition}
Let $U$ be to a poset such that
$$\displaystyle U = \left\{u, v, q \right\}$$ 
with ordering relation defined such that $v, w \le_U u$. Figure~\ref{U} shows the poset $U$. 
\end{definition}

\begin{figure}[h]
\caption{}
\label{U}
\includegraphics{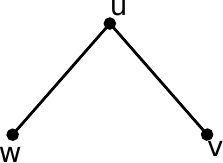}
\end{figure}

\begin{prop}
\label{prdp}
Let $P$ be a face poset with prism $Pr(P)$. Then, 
$$\displaystyle Pr(P)_{\ge 0} \cong P_{\ge 0} \times U.$$
\end{prop}

\begin{proof}
Define a bijection $\psi: P_{\geq 0}\times U\rightarrow Pr(P)$ by
\[\psi:\left\{\begin{array}{ccc}
(p, u)&\mapsto&p' \\
(p, v)&\mapsto&p_1\\
(p, w)&\mapsto&p_2
\end{array}\right..\]
We wish to show that $\psi$ is a bijection.  To do so, we need to show that $(p, s) \le (q,t)$ if and only if $\psi((p,s)) \le \psi((q,t)).$ 

For $s = t = u$, we need to show that $(p, u) \le (q, u)$ if and only if $\psi((p,u)) \le \psi((q,u))$. We know that  $(p, u) \le (q, u)$ if and only if $p \le_{P_{\ge 0}} q$. By Lemma~$\ref{bijj}$, we have $p \le_{P_{\ge 0}} q$ if and only if $p' \le_{\beta_P} q'$. For $p', q' \in (\beta_P)_{\ge 0}$, we have $p' \le_{(\beta_P)_{\ge 0}} q'$ if and only if $p' \le_{Pr(P)_{\ge 0}} q'$. Since $\psi((p,u)) = p'$ and $\psi((q,u)) = q'$, we know that $p' \le_{Pr(P)_{\ge 0}} q'$ if and only if $\psi((p,u)) \le_{Pr(P)_{\ge 0}} \psi((q,u))$. 

For $s = t = v$ , we need to show that $(p, v) \le (q, v)$ if and only if $\psi((p,v)) \le \psi((q,v))$. We know that $(p, v) \le_{P_{\ge 0} \times U} (q, v)$ if and only if $p \le_{P_{\ge 0}} q$. Since $P_{\ge 0}$ is isomorphic to $(P_1)_{\ge 0}$, we know $p \le_{P_{\ge 0}} q$ if and only if $p_1 \le_{(P_1)_{\ge 0}} q_1$. For $p_1, q_1 \in (P_1)_{\ge 0}$, we have $p_1 \le_{(P_1)_{\ge 0}} q_1$ if and only if $p_1 \le_{Pr(P)_{\ge 0}} q_1$. Since $\psi((p,r)) = p_1$ and $\psi((q,r)) = q_1$, we know that $p_1 \le_{Pr(P)_{\ge 0}} q_1$ if and only if $\psi((p,v)) \le \psi((q,v)).$ 

A similar argument shows that $(p, s) \le (q,t)$ if and only if $\psi((p,s)) \le \psi((q,t))$ in the case where $s = t = w$. 

For $s = u, t = v$, we want to show $(p, u) \le (q, r)$ if and only if $\psi((p,u)) \le \psi((q,v))$. Since $u > v$, there do not exist $p, q$ in $P_{\ge 0}$ such that $(p, u) \le (q,v)$. We will show there do not exist $p, q$ in $P_{\ge 0}$ such that $\psi((p,u)) \le \psi((q, v))$. We know that $\psi((p,u)) \le \psi((q,v))$ if and only if $p' \subseteq q_1$. Since $p' = p_1 \cup p_2$, there do not exist $p, q \in P_{\ge 0}$ such that $p' \subseteq q_1$. 

A similar argument shows that $(p, s) \le (q,t)$ if and only if $\psi((p,s)) \le \psi((q,t))$ in the case where $s =u, t=w$. 

For $s = v, t = u$, we need to show that $(p, v) \le (q, u)$ if and only if $\psi((p,v)) \le \psi((q,u))$. We know that $(p, v) \le (q, u)$ if and only if $p \subseteq q$. Since $p_1 \cap q_2 = \emptyset$ for all $p, q \in P_{\ge 0}$, we know that $p \subseteq q$ if and only if $p_1 \subseteq q_1 \cup q_2$. Since $q' = q_1\cup q_2$, we have $p_1 \subseteq q_1\cup q_2$ if and only if $p_1 \subseteq q'$. Since $\psi((p,v)) = p_1$ and $ \psi((q,u)) = q'$, we know that $p_1 \subseteq q'$ if and only if $\psi((p,v)) \le \psi((q,u))$. 

A similar argument shows that $(p, s) \le (q,t)$ if and only if $\psi((p,s)) \le \psi((q,t))$ in the case where $s =w, t=u$. 

For $s = v, t = w$, we need to show $(p, v) \le (q, w)$ if and only if $\psi((p,v)) \le \psi((q,w))$. Since $v$ is not below $w$ in the poset U, there does not exist  $p, q$ in $P_{\ge 0}$ such that $(p, v) \le (q,w)$. We want to show there does not exist  $p, q$ in $P_{\ge 0}$ such that $\psi((p,v)) \le \psi((q,w))$. We know that $\psi((p,v)) \le \psi((q,w))$ if and only if $p_1 \subseteq q_2$. However, for all $\left\{p, q \in P_{\ge 0}\mid p_1 \subseteq q_2 \right\}= \emptyset$. Hence, there do not exist $p, q \in P_{\ge 0}$ such that $p_1 \subseteq q_2$. 

A similar argument shows that $(p, s) \le (q,t)$ if and only if $\psi((p,s)) \le \psi((q,t))$ in the case where $s =w, t=v$. 

Hence, $\psi$ defines an isomorphism of posets and so $Pr(P)_{\ge 0} \cong P_{\ge 0} \times U$.

\end{proof}

\begin{theorem}
\label{prismform}
$$\displaystyle \mathcal{M}_{Pr(P)}  (z) = (3-2z) \cdot \left(\mathcal{M}_P (z) - g_{-1}^P (z)\right) + g_{-1}^{Pr(P)}.$$ 
\end{theorem}

\begin{proof}
By Proposition~$\ref{prdp}$ and Lemma~$\ref{dp}$, we know 
$$\displaystyle \mathcal{M}_{Pr(P)_{\ge 0}} (z) = \mathcal{M}_{P_{\ge 0}} (z) \cdot \mathcal{M}_U (z) = (3-2z) \cdot \mathcal{M}_{P_{\ge 0}} (z).$$
Note that $\mathcal{M}_{P_{\ge 0}} (z) = \mathcal{M}_P (z) - g_{-1}^P (z)$ and $\mathcal{M}_{Pr(P)} (z)  = \mathcal{M}_{Pr(P)_{\ge 0}} (z) + g_{-1}^{Pr(P)} (z)$. 
Hence, $$\displaystyle \mathcal{M}_{Pr(P)} (z) = (3-2z) \cdot \left(\mathcal{M}_P (z) - g_{-1}^P (z)\right) + g_{-1}^{Pr(P)} (z).$$
\end{proof}

We now use this theorem to give a proof of the following result, originally proved in 2012 by Colleen Duffy, Wai Shan Chan, and Cary Schneider at the University of Wisconsin, Eau Claire:

\begin{cor}
\label{hypercube}
If $\mathpzc{C}^d$ is a $d$-dimensional hypercube with face poset $C^d$, then
$$\displaystyle \mathcal{M}_{C^d} (z) = (-2z+3)^d - z(2-z)^d + 1.$$
\end{cor}

We will use the following lemma in our proof:
\begin{lemma}
\label{botpoly}
If $\mathpzc{C}^d$ is a $d$-dimensional hypercube with face poset $C^d$, then
$$\displaystyle g_{-1}^{C^d} (z) = 1 - z(2-z)^d.$$
\end{lemma}
\begin{proof}
By Lemma~$\ref{facepoly}$, we know that
$$\displaystyle g_{-1}^{C^d} (z) = \sum_{r=-1}^d f_r \cdot (-z)^{r+1}.$$
Notice that $f_{-1} = 1.$ We need to calculate $f_r$ for $r \ge 0$ in terms of $d$. Notice that the convex hull of the set of points $(\pm 1, \pm 1,..., \pm 1)$ in $\mathbb{R}^d$ corresponds to a $d$-dimensional hypercube. We can count the number of $r$-dimensional faces $\mathpzc{f}_r$ as follows: \\
There are  ${d \choose r}$ ways to choose the $r$ dimensions of the hyperplane. Consider the set of vertices of the hypercube contained in the hyperplane. There are $2^{d-r}$ ways to choose the coordinates of the remaining $d-r$ dimensions that must remain fixed over all points in the hyperplane. Notice that there is only way to fill in the remaining $r$ coordinates which must range over all possible combinations of $\pm 1$. Thus, $\mathpzc{f}_r = 2^{d-n} {d \choose r}.$ 
Hence, we know that 
\begin{align*}
g_{-1}^{C^d} (z) &= 1 + \sum_{r=0}^{d} 2^{d-r} {d \choose r} (-z)^{r+1} \\
&= 1 - z \cdot \sum_{r=0}^d 2^{d-r} {d \choose r} (-z)^r \\
&= 1 - z(2-z)^d.  
\end{align*}
\end{proof}
\begin{proof}[Proof of Corollary~$\ref{hypercube}$]
We proceed by induction on $d$. For $d=0$, notice that $C^0 \cong T_2$. Hence, $\mathcal{M}_{C^0} = 2-z.$ \\
Now assume that $d > 0$ and that the result holds for hypercubes of smaller dimension. Notice that a hypercube of dimension $d$ is the prism over a hypercube of dimension $d-1$. Hence, we know that
$$\displaystyle \mathcal{M}_{C^d} (z) = (3-2z) \cdot \left(\mathcal{M}_{C^{d-1}}(z) - g_{-1}^{C^{d-1}} (z)\right) + g_{-1}^{C^d}.$$
By Lemma~$\ref{botpoly}$ and the induction hypothesis, we know that 
\begin{align*}
\mathcal{M}_{C^d} (z) &= (3-2z) \cdot \left((-2z+3)^{d-1} - z(2-z)^{d-1} + 1 - (1 - z(2-z)^{d-1})\right) \\
&+ 1 - z(2-z)^d \\
&= (-2z+3)^{d} - z(2-z)^d +1. 
\end{align*}
Thus, the result holds for any hypercube $\mathpzc{C}^d$ with face poset $C^d$ for any $d \ge 0$. 
\end{proof}

\section{Gluing Together Convex Polytopes}

Given $d$-dimensional convex polytopes $\mathpzc{P}$,$\mathpzc{Q}$, and $(d-1)$-dimensional convex polytope $\mathpzc{R}$, such that $\mathpzc{P}$ and $\mathpzc{Q}$ contain isomorphic copies $\mathpzc{R}_P$ and $\mathpzc{R}_Q$ of $\mathpzc{R}$, respectively, we can obtain the $d$-dimensional convex polytope $\mathpzc{GL(P,Q, R)}$ as follows:

Glue $\mathpzc{P}$ and $\mathpzc{Q}$ along their isomorphic copies of $\mathpzc{R}$, so that $\mathpzc{R}_P$ and $\mathpzc{R}_Q$ line up. Then delete the faces $\mathpzc{R}_P$ and $\mathpzc{R}_Q$ which are now in the interior of the convex polytope.

This operation does not allow for a face $\mathpzc{f} \in \mathpzc{P}$ and a face $\mathpzc{g} \in \mathpzc{Q}$ to be at angles such that, after the gluing, $\mathpzc{f}$ and $\mathpzc{g}$ become one face. For example, the following construction is not permitted:

Let $\mathpzc{P}$ and $\mathpzc{Q}$ be cubes of side length $1$. Let $\mathpzc{R}$ be a square of side length $1$. Then the gluing process yields a rectangular prism. However, a square face in $\mathpzc{P}$ and a square face in $\mathpzc{Q}$ formed a single rectangular face in the gluing. This is not permitted.

We use the following definitions in this section:

\begin{definition}
For a  poset $P$, we define the poset $\hat{P}$ to be $P \cup \left\{\hat{1}^{\hat{P}}\right\}$ with ordering relation given by
\begin{itemize}
\item[(i)] For $p, q \in P$, $p \le_{\hat{P}} q$ if $p \le_P q$. 
\item[(ii)] For $p \in P$, $p \le_{\hat{P}} \hat{1}^{\hat{P}}$. 
\end{itemize}
\end{definition}

\begin{definition}
For a poset $P$ with a unique maximal element $\hat{1}^P$, we define the poset $\dot{P}$ to be $\hat{P} \setminus \left\{\hat{1}^P \right\}$ with ordering relation given by
\begin{itemize}
\item[(i)] For $p, q \in P \setminus \left\{\hat{1}^P \right\}$, $p \le_{\dot{P}} q$ if $p \le_{P} q$. 
\item[(ii)] For $p \in P$, $p \le_{\dot{P}} \hat{1}^{\hat{P}}$. 
\end{itemize}
\end{definition} 

\begin{definition}
Given face posets $P$, $Q$ of rank $d$, and $R$ of rank $d-1$, such that $P$ and $Q$ contain isomorphic copies $R_P$ and $R_Q$ of $R$, with isomorphism maps $\phi^P: R \rightarrow R^P$ and $\phi^Q: R \rightarrow R^Q$ define the gluing poset of $P$, $Q$ along $R$ to be 
$$\displaystyle Gl(P,Q,R) = (P \setminus \hat{R}_P) \cup (Q \setminus \hat{R}_Q) \cup \dot{R}.$$
where $\hat{R}_P = R_P \cup \left\{\hat{1}^P\right\}$ and where $\hat{R}_Q = R_Q \cup \left\{\hat{1}^Q\right\}$. 

For ease of notation, we will refer to $Gl(P, Q, R)$ as $Gl$ when there is no risk of ambiguity. 

We define the ordering relation $\le_{Gl}$ as follows: 
\begin{itemize}
\item[(i)]For $p, q \in P \setminus \hat{R}_P, p \le_{Gl} q$ if $p \le_P q$.  
\item[(ii)]If $p, q \in Q \setminus \hat{R}_Q, p \le_{Gl} q$ if $p \le_Q q$.  
\item[(iii)]If $p, q \in \dot{R}$ and $p \le_{Gl} q$ if $p \le_{\dot{R}} q$.
\item[(iv)]If $p \in P \setminus \hat{R}_P$ and $q \in \dot{R}, p \le_{Gl} q$ if $p \le_P \phi^P(q)$. 
\item[(v)]If $p \in \dot{R}$ and $q \in P \setminus \hat{R}_P, , p \le_{Gl} q$  if $\phi^P(p) \le_P q$. 
\item[(vi)]If $p \in Q \setminus \hat{R}_Q$  and $q \in \dot{R}, p \le_{Gl} q$  if $p \le_Q \phi^Q(q)$. 
\item[(vii)]If $p \in \dot{R}$  and $q \in Q \setminus \hat{R}_Q, p \le_{Gl} q$  if $\phi^Q(p) \le_Q q$. 
\end{itemize}

\end{definition}

Notice that $\hat{1}^{\dot{R}} = \hat{1}^{\hat{R}} = \hat{1}^{Gl}$, and that by construction, $Gl(P,Q,R)$ is the face poset of the convex polytope $\mathpzc{GL(P,Q,R)}$.

\begin{theorem}
\label{glue}
$$\displaystyle \mathcal{M}_{Gl} (z) = \mathcal{M}_P (z) + \mathcal{M}_Q (z) -  \mathcal{M}_R (z) - 1 +z - (1-z) \cdot g_{|R|}^R (z).$$
\end{theorem}

We will prove this using the following lemmas: 
\begin{lemma}
\label{hat}
$$\displaystyle \mathcal{M}_{\hat{R}} (z) = \mathcal{M}_R (z) + 1 - z.$$ 
\end{lemma}

\begin{proof}
We have 
$$\displaystyle \mathcal{M}_{\hat{R}} (z) = \mathcal{M}_R (z) + \mu_z^{\hat{R}} [\hat{1}^R, \hat {1}^{\hat{R}}] + \sum_{p \in R, p \neq \hat{1}^R} \mu_z^{\hat{R}} [p, \hat{1}].$$

We have $$\displaystyle \mu [p, \hat{1}] = -\sum_{p \leq q, q \in R} \mu [p,q] = \mu {p \leq x \leq \hat{1}^R} [p,x] = 0.$$

This means that $$ \displaystyle \mu_z^{\hat{R}} [p, \hat{1}] = 0.$$

Note that $$\displaystyle \mu_z^{\hat{R}} [\hat{1}^R, \hat{1}] = 1-z.$$

This means that

$$\displaystyle \mathcal{M}_{\hat{R}} (z) = \mathcal{M}_R (z) + 1 - z.$$ 
\end{proof}

\begin{lemma}
\label{dot}
$$\displaystyle \mathcal{M}_{\dot{R}} (z) = \mathcal{M}_{\hat{R}} (z) - (1-z) \cdot g_{|R|}^R(z)$$
\end{lemma}

\begin{proof}
We have 

$$\displaystyle \mathcal{M}_{\dot{R}} (z) = \mathcal{M}_{\hat{R}} (z) - \sum_{p \in R} \mu_z^R [p, \hat{1}^R]  + E$$ 
where $E$ is the effect of deleting $\hat{1}^R$ on $\mu_z^{Gl(P, Q, R)} [p, \hat{1}^{Gl(P,Q,R)}]$ for all $p \in R$.

We will begin by proving 
$$\displaystyle E =  \sum_{p \in R} z^{|\hat{1}^{Gl(P,Q, R)}| - |p|} \cdot \mu [p, \hat{1}^R]$$

We know that in $\dot{R}$, $\mu_g [p, \hat{1}^{Gl(P,Q,R)}] = -\sum_{p \leq q < \hat{1}^{Gl(P,Q,R)}} z^{|\hat{1}| - |p|} \cdot \mu [p, q]$. Take $q \in [p, \hat{1}^{Gl(P,Q,R)}]$ in the poset $\hat{R}$ such that $q \neq \hat{1}^R$. We know that $q < \hat{1}^{Gl(P,Q,R)}$ in $\dot{R}$, so each term $z^{|\hat{1}^{Gl(P,Q,R)}| - |p|} \cdot \mu [p, q]$ is present in the sum, except for $z^{|\hat{1}^{Gl(P,Q,R)}| - |p|} \cdot \mu [p, \hat{1}^R]$. Hence, $z^{|\hat{1}^{Gl(P,Q,R)}| - |p|} \cdot \mu [p, \hat{1}^R]$ must be added back for all $p \in R$.

\begin{align*}
\mathcal{M}_{\dot{R}} (z) &= \mathcal{M}_{\hat{R}} (z) - \sum_{p \in R} \mu_z^R [p, \hat{1}^R]  + \sum_{p \in R} z^{|\hat{1}^{Gl(P,Q,R)}| - |p|} \cdot \mu [p, \hat{1}^R]\\
&= \mathcal{M}_{\hat{R}} (z) - \sum_{p \in R} \mu_z^R [p, \hat{1}^R] + z \cdot \sum_{p \in R} z^{|\hat{1}^R| - |p|} \cdot \mu [p, \hat{1}^R] \\
&= \mathcal{M}_{\hat{R}} (z) - (1-z) \cdot \sum_{p \in R} \mu_z^R [p, \hat{1}^R] \\
&= \mathcal{M}_{\hat{R}} (z) - (1-z) \cdot g_{|R|}^R (z). 
\end{align*}
\end{proof}

\begin{proof}[Proof of Theorem~$\ref{glue}$]
By the definition of $Gl(P, Q, R)$, 
\begin{align*}
\mathcal{M}_Gl (z) &= \mathcal{M}_P (z)  - \mathcal{M}_{\hat{R}_P} (z) + \mathcal{M}_Q (z) - \mathcal{M}_{\hat{R}_Q} (z) + \mathcal{M}_{\dot{R}} (z) \\
&= \mathcal{M}_P (z) + \mathcal{M}_Q (z) - 2 \cdot \mathcal{M}_{\hat{R}} (z) + \mathcal{M}_{\dot{R}} (z).
\end{align*}
If we apply Lemma~$\ref{hat}$ and Lemma~$\ref{dot}$, we have
\begin{align*}
\mathcal{M}_Gl (z) &= \mathcal{M}_P (z)  + \mathcal{M}_Q (z) - 2 \cdot \mathcal{M}_{\hat{R}} (z) + \mathcal{M}_{\hat{R}} (z) - (1-z) \cdot g_{|R|}^R (z) \\
&= \mathcal{M}_P (z) + \mathcal{M}_Q (z) -  \mathcal{M}_{\hat{R}} (z) - (1-z) \cdot g_{|R|}^R (z)\\
&= \mathcal{M}_P (z) + \mathcal{M}_Q (z) -  \mathcal{M}_R (z) - 1 +z - (1-z) \cdot g_{|R|}^R (z). \\
\end{align*}
\end{proof}

\section{Simplicial Polytopes}

\begin{definition}
We define a boolean lattice of $n$ elements, denoted $2^{[n]}$ to be the set of subsets of $\left\{1, 2, ..., n\right\}$ ordered by set inclusion.  
\end{definition}

\begin{definition}
We define a simplicial poset $P$ to be any ranked poset such that every interval $I \in I(P)$ is a boolean lattice.
\end{definition}

\begin{theorem}
\label{simp}
In a simplicial poset $P$ with rank $d$, a unique minimum element $\hat{0}$, and $f$-vector $(f_{-1}, f_0,...,f_{d})$
$$\displaystyle \mathcal{M}_P (z) = \sum_{-1 \le r \le d} (1-z)^{r +1} f_r.$$
\end{theorem}

We will use the following lemma in our proof.
\begin{lemma}
\label{boolpoly}
In a boolean lattice, $2^{[n]}$ with rank $n-1$, 
$$\displaystyle g_{n-1}^{2^{[n]}} (z) = (1-z)^n$$
\end{lemma}

\begin{proof}
Since $S^{n-1} \cong 2^{[n]}$, $2^{[n]}$ is isomorphic to a face poset.
Hence, by Lemma~$\ref{facepoly}$, if $(f_{-1}, f_0,...f_{n-1})$ is the $f$-vector of $2^{[n]}$, then 
$$\displaystyle g_{n-1}^{2^{[n]}} (z) = \sum_{r=-1}^{n-1} f_r (-z)^{n-1-r}.$$
Notice that $f_r = {n \choose {r+1}}$. It follows that
$$\displaystyle g_{n-1}^{2^{[n]}} (z) = \sum_{r=-1}^{n-1} {n \choose r+1}  (-z)^{n-(r+1)}.$$
Applying the binomial theorem, we obtain
$$\displaystyle g_{n-1}^{2^{[n]}} (z) = (1-z)^n$$
\end{proof}

\begin{proof}[Proof of Theorem~$\ref{simp}$]

Notice that 
$$\displaystyle \mathcal{M}_{P} (z) = \sum_{p \in P} \sum_{q \le p} \mu_z [q, p].$$

Take $p \in P$ with rank $r$. Notice that
$$\displaystyle [\hat{0}, p] = \left\{q \in P_{\ge d-1} \mid q \le p \right\}.$$
Hence,  

\begin{align*}
\sum_{q \le p} \mu_z [q, p] &= \sum_{q \in [\hat{0},p]} \mu_z [q, p] \\
&= g^{[\hat{0},p]}_{|p|} (z)
\end{align*}

By the definition of a simplicial poset, we know $[\hat{0}, p] \cong 2^{[r+1]}$. Hence, by Lemma~$\ref{boolpoly}$, we have

$$\displaystyle \sum_{q \le p} \mu_z [q, p] = (1-z)^{r + 1}.$$

It follows that,

$$\displaystyle \mathcal{M}_{P} (z) = \sum_{p \in P} (1-z)^{r + 1}.$$

Grouping the elements by rank, we obtain

$$\displaystyle \mathcal{M}_{P} (z) = \sum_{-1 \le r \le d} (1-z)^{r + 1} f_r.$$
\end{proof}

\begin{definition}
We define a simplicial polytope to be a $d$-dimensional convex polytope $\mathpzc{P}$ such that any $(d-1)$-dimensional face is a simplex. 
\end{definition}

\begin{definition}
We define a near-simplicial poset $P$ to be any ranked poset such that $P_{\le |P|-1}$ is a disjoint union of simplicial posets with unique minimum elements. 
\end{definition}

Notice that the face poset $P$ of a simplicial polytope $\mathpzc{P}$ is near-simplicial.  

\begin{theorem}
\label{nearsimp}
In a near-simplicial poset $P$ with rank $d$ and $f$-vector $(f_{-1}, f_0,...f_{d})$
$$\displaystyle \mathcal{M}_P (z) = g^P_{d} (z) + \sum_{-1 \le r \le d -1} (1-z)^{r +1} f_r.$$
\end{theorem}

\begin{proof}

Notice that $\mathcal{M}_P (z) = g^P_{d} (z) + \mathcal{M}_{P_{\le d-1}} (z)$. First, we consider $\mathcal{M}_{P_{\le d-1}} (z)$. Since $P_{\le d-1}$ is a disjoint union of simplicial posets with unique minimum elements, we can obtain $\mathcal{M}_{P_{\le d-1}}$ by taking the sum of the M\"{o}bius polynomials of each of these simplicial posets. By Theorem~$\ref{simp}$, 
$$\displaystyle \mathcal{M}_{P_{\le d-1}} (z) = \sum_{-1 \le r \le d-1} (1-z)^{r + 1} f_r.$$

It follows that,

$$\displaystyle \mathcal{M}_P (z) = g^P_{d} (z) + \sum_{-1 \le r \le d -1} (1-z)^{r +1} f_r.$$

\end{proof}

In a face poset $P$ with rank $d$, by Lemma~$\ref{facepoly}$, we know that $g_{d}^{P} (z) = \sum_{r=-1}^{d} f_r (-z)^{d-r}$. Hence, we can simplify Theorem~$\ref{simp}$ for the face posets of simplicial polytopes. 

\begin{cor}
Given a $d$-dimensional simplicial polytope $\mathpzc{P}$ with face poset $P$, 

$$\displaystyle \mathcal{M}_P (z) = \sum_{r=-1}^{d} f_r (-z)^{d-r} + \sum_{-1 \le r \le d -1} (1-z)^{r +1} f_r.$$
\end{cor}

\section{Eulerian Posets}

\begin{definition}
An Eulerian poset is a ranked poset $P$ with unique maximum and minimum elements such that, for all $p, q \in P$ with $p \le q$,
$$\displaystyle \mu [p,q] = (-1)^{|q| - |p|}.$$ 
\end{definition}

Note that any interval of an Eulerian poset is itself an Eulerian poset. By Lemma~$\ref{Faceuler}$, we know all face posets of convex polytopes are Eulerian posets. In this section, we will discuss the M\"{o}bius polynomials of Eulerian posets. Fix an Eulerian poset $P$ of rank $d$. We define $N_{i,j}$ to be the number of intervals $[p, q]$ such that $p$ is an element of rank $i$ and $q$ is an element of rank $j$ for $-1 \le i \le j$. 
We can calculate the M\"{o}bius polynomial of $P$ in terms of the $N_{i, j}$ values:
$$\displaystyle \mathcal{M}_P (z) = \sum_{0 \le l \le n+1} \sum_{j-i = l} N_{i, j} (-z)^{l}$$ 

Hence, to calculate the M\"{o}bius polynomial of an Eulerian poset $P$, it suffices to calculate $N_{i,j}$ for $-1 \le i \le j \le |P|$.
We know that $N_{-1, r}, N_{r, |P|},$ and $N_{r, r}$ are all equal to $f_r$. 
It remains to calculate $N_{i,j}$ for $0 \le i < j < d$. We will use the following lemmas: 

\begin{lemma}
\label{eulerchar}
For an Eulerian poset $P$ with rank $d$ and $f$-vector 

$(f_{-1}, f_0,..., f_d)$, we have 
$$\displaystyle \sum_{i=0}^{d-1} f_i \cdot (-1)^i = (-1)^{d + 1} + 1$$
\end{lemma}

\begin{proof}
By Lemma~\ref{Muiszero}
$$\displaystyle \sum_{\hat{0} \le p \le \hat{1}} \mu [\hat{0}, p] = 0.$$
Separating the elements by
 rank, we get: 
$$\displaystyle \sum_{i= -1}^d \sum_{p \in P_i} \mu [\hat{0}, p] = 0.$$
Since $\mu [\hat{0}, p]$ is $(-1)^{|p| + 1}$, we get: 
$$\displaystyle \sum_{i= -1}^d  (-1)^{i+1} f_i = 0.$$
Moving the $i = -1$ and $i = d$ terms to the right side, we get:
$$\displaystyle \sum_{i= 0 }^{d-1}  (-1)^{i+1}  f_i = (-1)[(-1)^{d +1} + 1].$$
Hence,
$$\displaystyle \sum_{i= 0 }^{d-1} (-1)^{i} f_i  = (-1)^{d +1} + 1.$$
\end{proof}

\begin{lemma}
\label{System}
For an Eulerian poset $P$ with rank $d$, and $f$-vector 

$(f_{-1}, f_0...f_d)$,  we have the following system of $2d-2$ equations: 

\begin{itemize}
\item[(i)] For $1 \le j \le d-1,$
\[\sum_{i=0}^{j-1} (-1)^i N_{i, j}  = ((-1)^{j+1} + 1) f_j.\]
\item[(ii)] For $0 \le j \le d-2$,
\[\sum_{j=i+1}^{d-1} (-1)^{j-i-1} N_{i, j} = ((-1)^{d-i} + 1) f_i\]
\end{itemize}
\end{lemma}

\begin{proof}

We begin by proving (i). For an element $p \in P$, let $B_i^p$ be the number of elements of rank $i$ below $p$. That is, 
$$\displaystyle B_{i, p} = |P_i \cap \left\{q \in P \mid q \le p \right\}|.$$
Notice that $\sum_{p \in P_j} B_{i, p} = N_{i,j}$. Thus, we have: 

\begin{align*}
\sum_{i=0}^{j-1} (-1)^i N_{i, j} &= \sum_{i=0}^{j-1} \sum_{p \in P_j} (-1)^i B_{i, p} \\
&= \sum_{p \in P_j} \sum_{i=0}^{j-1} (-1)^i B_{i,p}.
\end{align*}

We know that the interval $[\hat{0}, p]$ is an Eulerian poset for all $p \in P$. Thus for all $p \in P$, Lemma~$\ref{eulerchar}$ gives us
$$\displaystyle \sum_{i=0}^{j-1} (-1)^i B_{i, p} = (-1)^{j+1} + 1.$$
Thus, we have
\begin{align*}
\sum_{i=0}^{j-1} (-1)^i N_{i, j} &= \sum_{p \in P_j} \left[(-1)^{j+1} + 1 \right] \\
&= \left[(-1)^{j+1} + 1 \right] f_j.
\end{align*}

The proof of (ii) is similar. For each $p \in P$, we define $A_{p, j}$ to be the number of elements of rank $j$ above $p$. That is,
$$\displaystyle A_{p, j} = |P_j \cap \left\{q \in P \mid q \ge p \right\}|.$$
Notice that $\sum_{p \in P_i} A_{p, j} = N_{i, j}$. Thus, we have
\begin{align*}
\sum_{j=i+1}^{d-1} (-1)^{j-i-1} N_{i, j} &= \sum_{j=i+1}^{d-1} \sum_{p \in P_i} (-1)^{j-i-1} A_{p, j} \\
&= \sum_{p \in P_i} \sum_{j=i+1}^{d-1} (-1)^{j-i-1} A_{p,j}. 
\end{align*}
Since the interval $[p, \hat{1}]$ is Eulerian, Lemma~$\ref{eulerchar}$ gives us  gives us
$$\displaystyle \sum_{j=i+1}^{d-1} (-1)^{j-i-1} A_{p, j} = (-1)^{d-i} + 1.$$
It follows that
\begin{align*}
\sum_{j=i+1}^{d-1} (-1)^{j-i-1} N_{i, j} &= \sum_{p \in P_i} \left[(-1)^{d-i} + 1 \right] \\
&= \left[(-1)^{d-i} + 1 \right] f_i.
\end{align*}
\end{proof}

Notice that of the equations in (i) of Lemma~$\ref{System}$, 
$$\displaystyle \sum^{d-1}_{i=0} N_{i,d-1} = (-1)^d +1$$
is a linear combination of the others.
We will show that the remaining $(2d-3)$ equations are linearly independent by solving these equations in the Proof of Theorem~$\ref{GenMob}$.

In the case where $d=3$, we can calculate the $N_{i,j}$ from $f_0, f_1,$ and $f_2$. This gives us the following: 

\begin{lemma}
\label{three}
For an Eulerian poset $P$ with rank $3$ and $f$-vector 

$(f_{-1}, f_0, f_1, f_2, f_3)$, 
$$\displaystyle \mathcal{M}_P (z) = (2+f_0+f_1+f_2) - (f_0+4f_1+f_2)*z + 4f_1 z^2 - (f_0+f_2) z^3 + z^4.$$ 
\end{lemma}

\begin{proof}
By Lemma~$\ref{System}$ we have three equations in three variables: 
\begin{align*}
 N_{0,1} &= 2f_1 \\
 N_{0,2} - N_{1,2} &= 0 \\
 N_{0,1} - N_{0,2} &= 0. \\
\end{align*}
Solving, we obtain $N_{0,1} = 2f_1, N_{0,2} = 2f_1, N_{1,2} = 2f_1$. We know 
$$\displaystyle \mathcal{M}_P (z)= \sum_{0 \le l \le 4} \sum_{j-i = l} N_{i, j} \cdot (-z)^{l}.$$
We will calculate $\sum_{j-i = l} N_{i, j}$ for $l$ from $0$ to $4$. 

For $l=0$, we have
\begin{align*}
\sum_{j-i = 0} N_{i, j} &= N_{-1, -1} + N_{0,0} + N_{1,1} + N_{2,2} + N_{3,3} \\
&= f_{-1}+f_0+f_1+f_2 + f_{-3} \\
&= 2 + f_0 + f_1 + f_2. 
\end{align*}

For $l=1$, we have
\begin{align*}
\sum_{j-i = 1} N_{i, j} &= N_{-1,0} + N_{0,1} + N_{1,2} + N_{2,3} \\
&= f_0 + 2f_1 + 2f_1 + f_2 \\
&= f_0 + 4f_1 + f_2. 
\end{align*}

For $l=2$, we have
$$\displaystyle \sum_{j-i = 2} N_{i, j} = N_{-1,1} + N_{0,2} + N_{1,3} = f_1 + 2f_1 + f_1 = 4f_1.$$ 

For $l=3$, we have
$$\displaystyle \sum_{j-i = 3} N_{i, j} = N_{-1,2} + N_{0,3} = f_0 + f_2.$$

For $l=4$, we have
$$\displaystyle \sum_{j-i = 4} N_{i, j} = N_{-1,3} = 1.$$

It follows that:
$$\displaystyle \mathcal{M}_P (z) = (2+f_0+f_1+f_2) - (f_0+4f_1+f_2)*z + 4f_1 z^2 - (f_0+f_2) z^3 + z^4.$$ 
\end{proof}

For Eulerian posets of rank greater than $3$, we cannot calculate the M\"{o}bius polynomial in terms of the the $f_r$ values. To see that this is the case, consider the following example:

By Steinitz' Lemma \cite{stein}, we know that we can construct a convex polytope $\mathpzc{K}$ with $\mathpzc{f}$-vector $(1, 18, 38, 22, 1)$. Let $\mathpzc{P}$ be the convex polytope given by $\mathpzc{P} = \mathpzc{Py (K)}$. 

Let
$$\displaystyle\mathpzc{Q} = \mathpzc{Gl(Gl(Gl(C^4, Py(C^3), C^3), Py(C^3), C^3), Py(C^3), C^3)}.$$

By construction, both $\mathpzc{P}$ and $\mathpzc{C^i}$ have the same $\mathpzc{f}$-vector: 
$$\displaystyle (1, 19, 56, 60, 23, 1).$$

Using Theorem~$\ref{three}$ and Theorem~$\ref{pyramid}$, we can show that
\begin{align*}
\mathcal{M}_P (z) &= (2 - z) (80 - 192 z + 152 z^2 - 40 z^3 + z^4) \\
&= 160 - 464 z + 496 z^2 - 232 z^3 + 42 z^4 - z^5.
\end{align*} 

Using the formulas in Theorem~$\ref{glue}$, Theorem~$\ref{pyramid}$, and Theorem~$\ref{hypercube}$, we find that $\mathcal{M}_Q (z)$ is equal to
\begin{align*}
&(-2z + 3)^4 - z (2-z)^4 + 1 + 3 (2 - z) ((-2 z + 3)^3 - z (2 - z)^3 + 1) \\
&-3 ((-2 z + 3)^3 - z (2 - z)^3 + 1) - 3 (1 - z) \\
&- 3 (1 - z) (1 - 6 z + 12 z^2 - 8 z^3 + z^4 )
\end{align*}
which reduces to
 $$\displaystyle 160 - 478 z + 524 z^2 - 246 z^3 + 42 z^4 - z^5.$$
Hence, we have found two face posets with the same $f$-vector but different M\"{o}bius polynomials. This demonstrates that the M\"{o}bius polynomial cannot be calculated from the $f$-vector for an arbitrary Eulerian poset of rank greater than $3$. 

We will now choose $N_{i,j}$ values to be given along with the $f_r$ values to make the system in Lemma~$\ref{System}$ solvable. The other  $N_{i,j}$ values can be solved in terms of these given values. 

Assuming the equations in Lemma~$\ref{System}$ are linearly independent, we should be able to calculate the $\mathcal{M}_P (z)$ from the $f$-vector and some collection of all but $2d-3$ of the $N_{i, j}$. The following theorem shows that this is indeed possible:

\begin{theorem}
\label{GenMob}
Let $P$ be a rank-$d$ Eulerian poset. Given the $f$-vector $(f_{-1}, f_0,...,f_d)$ and the values $N_{i,j}$ for $0 < i+1 < j < d-1$, we have
\begin{align*}
\mathcal{M}_P (z) &= \sum_{i=-1}^d f_i - z (f_0 + 2f_1 + 2f_{d-2} + f_{d-1} + \sum_{i=0}^{d-3} ((-1)^{i} +1) f_{i+1}) \\
&+ (\sum_{0 \le k < i \le d-3} (-1)^{i+1-k} N_{k,i+1})) + \sum_{l=2}^{d-2} (-z)^l (f_{l-1} + f_{d-l} \\
&+ \sum^{d-2}_{j=l} (N_{j-l, j}) + ((-1)^{l-1} +1) f_{d-l-1} + ((-1)^{d-1} + (-1)^{l}) f_{d-l} \\
&+ \sum^{d-l-2}_{k=0} \left[(-1)^{d-k} N_{k,d-l}\right] + \sum^{d-2}_{j = d-l+1} ((-1)^{d-j} N_{d-l-1, j})) \\
&+ (-z)^{d-1} (((-1)^{d} + 1)f_0 +  f_1 + 2(-1)^{d-1} + 1) f_{1} + f_{d-2} \\
&+ \sum^{d-2}_{j = 2} \left[N_{0, j} (-1)^{d+j}\right]) + (-z)^{d} \left[f_{d-1} + f_0 \right] + (-z)^{d+1}.
 \end{align*}
\end{theorem}

\begin{proof}
We begin by solving for the missing $N_{i,j}$ values.

Part $(i)$ of Lemma~$\ref{System}$ gives us
$$\sum_{i=0}^{j-1} (-1)^i N_{i, j} = ((-1)^{j+1}+1) f_j$$
for $1 \le j \le d-1$. It follows that $N_{0,1} = 2f_1$, and that for $1 \le i \le d-2$
$$\displaystyle N_{i, i+1} = ((-1)^{i} +1) f_{i+1} + \sum^{i-1}_{k=0} (-1)^{i+1-k} N_{k,i+1}.$$

Part $(ii)$ of Lemma~$\ref{System}$ gives us
$$\displaystyle \sum_{j=i+1}^{d-1} (-1)^{j-i-1} N_{i, j} = ((-1)^{d-i}+1) f_i$$ 
for $0 \le j \le d-2$. This gives us 
$$\displaystyle N_{i, d-1} = ((-1)^{d-i} + 1) f_i + \sum^{d-2}_{j = i+1} (-1)^{d-j} N_{i,j}.$$
Using our values for $N_{i,i+1}$, we obtain
$$\displaystyle N_{0, d-1} = ((-1)^{d} + 1) f_0 + (-1)^{d-1} 2f_1 + \sum^{d-2}_{j = 2} (-1)^{d-j} N_{0,j}$$
and
\begin{align*}
N_{i, d-1} &= ((-1)^{d-i} + 1) f_i + ((-1)^{d-1} + (-1)^{d-i-1}) f_{i+1} \\
&+ \sum_{k=0}^{i-1} (-1)^{d-k} N_{k, i+1} + sum^{d-2}_{j = i+2} (-1)^{d-j} N_{i,j}
\end{align*}
for $0 < i <d-2$.
We calculate $N_{d-2, d-1}$ using $(ii)$
$$\displaystyle N_{d-2, d-1} = (-1)^{(d-1) - (d-2)} (-1) ((-1)^{d - (d-2)} + 1) f_{d-2} = 2f_{d-2}.$$
We will now use these values to calculate $\mathcal{M}_P (z)$. We know 
$$\displaystyle \mathcal{M}_P (z) = \sum_{0 \le l \le d+1} \sum_{j-i = l} N_{i, j} (-z)^{l}.$$
We will calculate $\sum_{j-i = l} N_{i, j}$ for $l$ from $0$ to $d+1$. \\

For $l=0$, we have
$$\displaystyle \sum_{j-i = 0} N_{i, j} = \sum_{i=-1}^{d} N_{i, i}\sum_{i=-1}^d f_i.$$

For $l=1$, we have
\begin{align*}
\sum_{j-i = 1} N_{i, j} &= N_{-1, 0} + N_{0,1} + \sum^{d-3}_{i=0} N_{i, i+1} + + N_{d-2, d-1} + N_{d-1,d} \\
&= f_0 + 2f_1 + 2f_{d-2} + f_{d-1} + \sum_{i=0}^{d-3} (((-1)^{i} +1) f_{i+1} \\ 
&+ \sum^{i-1}_{k=0} (-1)^{i+1-k} N_{k,i+1}) \\
&=f_0 + 2f_1 + 2f_{d-2} + f_{d-1} + \sum_{i=0}^{d-3} ((-1)^{i} +1) f_{i+1}) \\
&+ (\sum_{0 \le k < i \le d-3} (-1)^{i+1-k} N_{k,i+1}).
\end{align*}

For $2 \le l \le d-2$, we have
\begin{align*} 
\sum_{j-i = l} N_{i, j} &= N_{-1, l-1} + \sum^{d-2}_{j=l} N_{j-l, j} + N_{d-l-1, d-1} + N_{d-l, d} \\
&= f_{l-1} + f_{d-l} + \sum^{d-2}_{j=l} (N_{j-l, j}) + ((-1)^{l-1} +1) f_{d-l-1} + ((-1)^{d-1} \\
&+ (-1)^{l}) f_{d-l} + \sum^{d-l-2}_{k=0} \left[(-1)^{d-k} N_{k,d-l}\right] \\
&+ \sum^{d-2}_{j = d-l+1} ((-1)^{d-j} N_{d-l-1, j}).
\end{align*}

For $l = d-1$, we have
\begin{align*}
\sum_{j-i = d-1} N_{i, j} &= N_{-1, d-2} + N_{0, d-1} + N_{1, d} \\
&= f_1 + f_{d-2} + ((-1)^{d} + 1)f_0 + ((-1)^{d-1}) 2f_{1} \\
&+ \sum^{d-2}_{j = 2} \left[N_{0, j} (-1)^{d+j}\right] \\
&= ((-1)^{d} + 1)f_0 +  f_1  + (2(-1)^{d-1} + 1) f_{1} + f_{d-2} \\
&+ \sum^{d-2}_{j = 2} \left[N_{0, j} (-1)^{d+j}\right]. 
 \end{align*}
 
 For $l = d$, we have
$$\displaystyle \sum_{j-i = d} N_{i, j} = N_{-1, d-1} + N_{0, d} = f_{d-1} +
 f_0.$$

 For $l = d+1$, we have
$$\displaystyle \sum_{j-i = d+1} N_{i, j} = N_{-1, d} = 1.$$

 It follows that 
\begin{align*}
\mathcal{M}_P (z) &= \sum_{i=-1}^d f_i - z (f_0 + 2f_1 + 2f_{d-2} + f_{d-1} + \sum_{i=0}^{d-3} ((-1)^{i} +1) f_{i+1}) \\
&+ (\sum_{0 \le k < i \le d-3} (-1)^{i+1-k} N_{k,i+1})) + \sum_{l=2}^{d-2} (-z)^l (f_{l-1} + f_{d-l} \\
&+ \sum^{d-2}_{j=l} (N_{j-l, j}) + ((-1)^{l-1} +1) f_{d-l-1} + ((-1)^{d-1} + (-1)^{l}) f_{d-l} \\
&+ \sum^{d-l-2}_{k=0} \left[(-1)^{d-k} N_{k,d-l}\right] + \sum^{d-2}_{j = d-l+1} ((-1)^{d-j} N_{d-l-1, j})) \\
&+ (-z)^{d-1} (((-1)^{d} + 1)f_0 +  f_1 + 2(-1)^{d-1} + 1) f_{1} + f_{d-2} \\
&+ \sum^{d-2}_{j = 2} \left[N_{0, j} (-1)^{d+j}\right]) + (-z)^{d} \left[f_{d-1} + f_0 \right] + (-z)^{d+1}.
 \end{align*}
\end{proof}
\bibliographystyle{plain}
\bibliography{MobiusPolynomialsPaper}
\end{document}